\documentclass[11pt]{amsart}
\usepackage{hyperref}
\usepackage{amsfonts}
\usepackage{amsmath}
\usepackage{amssymb}
\usepackage{amstext}
\usepackage{amsthm}
\usepackage{mathpple}

\def\ram{{\rm ram}}
\def\branch{{\rm branch}}
\def\unbranch{{\rm u}}
\def\rram{{\rm r}}
\def\bbranch{{\rm b}}

\def\exc{{\rm ex}}
\def\inj{{\rm inj}}
\def\surj{{\rm surj}}

\def\sriso{\stackrel{\sim}{\rightarrow}}
\newcommand{\sra}[1]{\stackrel{#1}{\rightarrow}}

\def\tfae{(*)}

\def\aff{{\mathbb A}}
\def\ff{{\mathbb F}}
\def\proj{{\mathbb P}}
\def\rat{{\mathbb Q}}
\def\calb{{\mathcal B}}

\def\calo{{\mathcal O}}
\def\cals{{\mathcal S}}
\def\calx{{\mathcal X}}
\def\cross{\times}
\def\ra{\rightarrow}
\def\inject{\hookrightarrow}
\def\del{\partial}
\def\iso{\cong}
\newcommand{\ang}[1]{{{\langle #1 \rangle}}}
\renewcommand{\bar}[1]{{\overline{#1}}}
\newcommand{\abs}[1]{{\left|#1\right|}}
\newcommand{\st}[1]{\{#1\}}
\newcommand{\til}[1]{{\widetilde{#1}}}
\newcommand{\rest}[1]{|_{#1}}
\DeclareMathOperator{\aut}{Aut}
\DeclareMathOperator{\gal}{Gal}
\DeclareMathOperator{\spec}{Spec}
\def\inv{^{-1}}

\newenvironment{alphabetize}{\begin{enumerate}

}{\end{enumerate}}

\numberwithin{equation}{section}

\theoremstyle{plain}
\newtheorem{theorem}{Theorem}[section]
\newtheorem{condition}[theorem]{Condition}
\newtheorem{lemma}[theorem]{Lemma}
\newtheorem{proposition}[theorem]{Proposition}
\newtheorem{corollary}[theorem]{Corollary}

\newtheorem*{maintheorem}{Theorem \ref{thnrdc}}
\newtheorem*{surftheorem}{Corollary \ref{corsurf}}
\theoremstyle{remark}
\newtheorem{remark}[theorem]{Remark}
\begin{document}
\title{Exceptional covers of surfaces}
\author{Jeffrey D. Achter}
\email{j.achter@colostate.edu}
\address{Department of Mathematics\\Colorado State University\\Fort
  Collins, CO 80523}

\begin{abstract}
Consider a finite morphism $f:X \ra Y$ of smooth, projective varieties
over a finite field $\ff$.  Suppose $X$ is the vanishing locus in
$\proj^N$ of $r$ forms of degree at most $d$.  We show that there is a
constant $C$ depending only on $(N,r,d)$ and $\deg(f)$ such that if
$\abs{\ff}>C$, then $f(\ff): X(\ff) \ra Y(\ff)$ is injective if and
only if it is surjective.
\end{abstract}
\maketitle


\section{Introduction}

Consider a finite, generically \'etale morphism $f:X \ra Y$ between
smooth, projective varieties over a finite field $\ff$ of
characteristic $p$.  The cover $f$ is
called exceptional if the only geometrically irreducible component of
$X\cross_YX$ which is defined over $\ff$ is the diagonal.  Exceptional
covers have the following intriguing property: the induced map
$f(\ff): X(\ff) \ra Y(\ff)$ on $\ff$-points is bijective.  This
theorem, due to Lenstra, is proved in \cite{gtz}; we defer to that
article for the history of this circle of ideas.

In \cite{gtz}, Guralnick, Tucker and Zieve prove a partial converse
for projective curves. Specifically, they show that for fixed genus
$g=g(X)$ and degree $\deg(f)$, there exists an effective constant $C$
such that the following holds: 
if 
$\ff_q/\ff$ is an extension with $q>C$, and if $f(\ff_q)$ is injective,
then $f$ is exceptional.  (Note that this implies that $f$ is
bijective.)   They prove something like this in higher dimension (see
Remark \ref{remgtz} below), except that the constant $C$ is allowed to
depend on $X$, $Y$ and $f$.  They conjecture \cite[5.5]{gtz} that $C$
need only depend on $\deg(f)$ and the topology of $X$.

The calculation of $C$ relies on understanding the topology of the
cover $f$.
Indeed, if $Z$ is a nondiagonal component of $X\cross_YX$, and if
$f(\ff_q)$ is injective, then every point of $Z(\ff_q)$ is
actually a ramification point.  The dimension of $Z$ is greater
than that of the ramification locus $\ram(f)$.  Weil-type estimates show
that if $\abs{Z(\ff_q)} \le \abs{\ram(f)(\ff_q)}$, then $q$ must be
small relative to the Betti numbers of $Z$ and $\ram(f)$.
In the special case of curves, Guralnick, Tucker 
and Zieve obtain effective bounds for these Betti numbers, and thus a
bound for the constant $C$.  

Consider the following condition on a triple $(X/\ff, C, n)$
consisting of a smooth, projective, geometrically irreducible variety $X$, a
constant $C$, and a natural number $n \ge 2$:

\begin{condition}
\label{thecond}
Let $Y/\ff$ be a smooth projective geometrically irreducible variety, and let $f:X \ra Y$ be a
finite, tamely ramified, generically \'etale morphism of degree $n$.
Then
$(f:X \ra Y/\ff,C)$ satisfies \tfae:
\begin{quotation}
{\bf \tfae}\ \ \parbox{5in}
{
 If 
$\ff_q/\ff$ is a finite extension with $q>C$, then the following are
equivalent:
\begin{alphabetize}
\item $f(\ff_q):X(\ff_q) \ra Y(\ff_q)$ is injective;
\item $f(\ff_q):X(\ff_q) \ra Y(\ff_q)$ is surjective;
\item $f_{\ff_q}:X_{\ff_q} \ra Y_{\ff_q}$ is exceptional.
\end{alphabetize}
}
\end{quotation}
\end{condition}

The main purpose of this note is to explain the following
result.  Say that a projective variety $X$ is of type $(N,r,d)$ if $X$
is isomorphic to the vanishing in $\proj^N$ of at most $r$ homogeneous
forms, each of which has degree at most $d$. 

\begin{maintheorem}
Given data $(N,r,d)$ and $n$, there exists a constant $C$ so that the
following holds.  If $X$ is a smooth, projective geometrically
irreducible of type $(N,r,d)$, then $(X/\ff, C, n)$ satisfies
\eqref{thecond}.
\end{maintheorem}

The type $(N,r,d)$ is not intrinsic to $X$, but is rather an attribute
of $X$ along with a chosen embedding.  In fact,
\cite[Conjecture 5.5]{gtz}
asks for a result similar to \ref{thnrdc}, but which depends only
on the $\ell$-adic Betti numbers of the source variety.  Here is a
partial answer for surfaces:

\begin{surftheorem}
Given nonnegative integers $b_1$, $b_2$ and $b_3$ and a natural number
$n$, there exists an
effective constant $C$ so that the following holds.  Let $X$ be a
smooth, projective, geometrically irreducible surface over a finite
field $\ff$.   Suppose
that $X$ is of general type and admits a flat lifting to $W(\ff)/p^2$, and that $\dim
H^i(X,\rat_\ell) = b_i$ for $1 \le i \le 3$.  Then $(X/\ff,C,n)$ satisfies
\eqref{thecond}.
\end{surftheorem}

Results such as Theorem \ref{thnrdc} have long been known in the
special case where $X = Y = \proj^n$ \cite{fried74}.  For arbitrary varieties, 
\cite[Prop.\ 5.6]{gtz} shows the existence of a constant $C$ which depends on $f:X \ra
Y$ such that $(f:X\ra Y/\ff,C)$ satisfies \tfae.  The contribution of
the present note is to prove a more uniform version of these results.

The final section gives a new supply (Proposition \ref{propex}) of
examples of exceptional covers, as well as examples of covers which
are injective or surjective on $\ff$-points, but not bijective.  In
contrast to previously published examples, which tend to focus on
curves, projective spaces or abelian varieties, these covers involve
varieties of arbitrary dimension and arbitrarily intricate topology.

I  thank Guralnick for a helpful discussion of \cite[Prop. 5.6]{gtz};
see Remark \ref{remgtz}.

\section{Exceptional covers of polarized varieties}

Say that a projective variety $X$ admits a polarization of type $(N,r,d)$,
or simply that $X$ {\em has type} $(N,r,d)$, if there exists an embedding
$X\inject \proj^N$ such that $X$ can be expressed as the vanishing
locus of at most $r$ homogeneous forms, each of which has degree at
most $d$.  
Similarly, say that an affine variety $U$ {\em has type}
$(N,r,d)$ if there exists an embedding $U\inject \aff^N$ such that $U$
is the vanishing locus of at most $r$ polynomials, each of which has
degree at most $d$. 

Suppose $X$ is projective of type $(N,r,d)$, and that $X\inject \proj^N$ realizes this type.  If $H\subset\proj^N$ is a hyperplane, 
then $X\cap H$ is projective of type $(N-1,r,d)$, while $X-(X\cap H)$
is affine of type $(N,r,d)$.   For the
moment, we work over an arbitrary and suppressed field $k$.

\begin{lemma} 
\label{lemelimination}
Gven data $(N,r,d)$ and $n$ there exist effective constants $(\til N, \til r,
\til d)$ such that the following holds.  If $X$ is a geometrically irreducible
projective variety of type $(N,r,d)$,  
and if $f:X \ra Y$ is a finite map of degree $n$, then $Y$ is
projective of type $(\til N, \til r, \til d)$.
\end{lemma}

\begin{proof}

Embed $Y \inject \proj^M$, where $M = 2\dim Y+1$.  On an open affine
subset $U\subseteq X$, $f$ is represented by a map $(f_0, \cdots,
f_M)$, with each $f_j \in k[\aff^N]$.
By
B\'ezout's theorem, each $f_j$ has degree at most $n$.
Elimination theory, either in its classical form 
or its Gr\"obner-theoretic incarnation, 
yields an effective upper bound on the number
and degree of the equations necessary to define the Zariski closure of
$f(U)$ in $\aff^M$.
\end{proof}

\begin{lemma}
\label{lembranch}
Given data $(N,r,d)$ and $n$ there are effective constants $(N_\rram, r_\rram, d_\rram)$
and $(N_\bbranch, r_\bbranch, d_\bbranch)$ such that the following
holds.  If $X$ is a geometrically irreducible projective variety of type $(N,r,d)$ and if $f:X \ra Y$ is a
finite map of projective varieties of degree $n$, then the
ramification locus $\ram(f)\subset X$ is of type $(N_\rram,
r_\rram, d_\rram)$, and the branch locus $\branch(f)\subset Y$ is
of type $(N_\bbranch, r_\bbranch, d_\bbranch)$.
\end{lemma}

\begin{proof}
As in the proof of Lemma \ref{lemelimination}, represent $f$ locally by a morphism
$(f_1, \cdots, f_M): \aff^N \ra \aff^M$, and let $Jf = (\frac{\del
  f_i}{\del x_j})_{1 \le i \le M, 1 \le j \le N}$ be the Jacobian
  matrix of this morphism.  The ramification locus $\ram(f)$ is the
  locus of all $P\in U$ where the restriction of $Jf$ to $T_PU$ has
  rank less than $n$.  Therefore, $\ram(f)$ is the intersection of $U$
  and the vanishing locus of minors of a certain matrix,
and its type
  $(N_\rram, r_\rram, d_\rram)$ depends only on the type of $X$ and on
  the degree of $f$.

Lemma \ref{lemelimination}, applied to the cover $\ram(f) \ra
\branch(f)$, shows that $\branch(f)$ is of type 
$(N_\bbranch, r_\bbranch, d_\bbranch)$ for constants which depend only on
$(N,r,d,n)$.  
\end{proof}

Let $g:Z \ra Y$ be an \'etale cover of affine varieties.  
Say that $g$  is tamely ramified at the boundary if there are
compactifications $Z\inject \bar Z$ and $Y\inject \bar Y$, with $\bar
Z$ and $\bar Y$ projective, and a morphism $\bar g:\bar Z \ra
\bar Y$ compatible with $g$ which is at worst tamely ramified.  (This
notion is independent of the choice of compactification; moreover, in
the present context, the cover $g:Z \ra Y$ arises as an open subcover
of a known cover of projective varieties.)

\begin{lemma}
\label{lemgaloisbetti}
Given data $(N,r,d)$ there exists an effective constant $\sigma$ such that the
following holds.  Let $Y$ be a smooth geometrically irreducible affine variety of type $(N,r,d)$
over a field $k$, and let $g:Z \ra Y$ be an \'etale Galois cover
tamely ramified at the boundary.  If $\ell$ is any rational prime
invertible in $k$, then the sum of the compact
$\ell$-adic Betti numbers of $Z$ satisfies
\[
\sigma_c(Z,\rat_\ell) := \sum_{i=0}^{\dim Z} \dim H^i_c(Z\cross \bar k, \rat_\ell)
\le \deg(g) \sigma.
\]
\end{lemma}

\begin{proof}
This is \cite[Prop.\ 4.5]{kowalskisieve};  the effective constant
of {\em loc. cit.} depends only on the type of $Y$.
\end{proof}

\begin{lemma}
\label{lemunbranched}
Given data $(N,r,d)$ and $n$ there exist effective constants $(N_\unbranch, r_\unbranch,
d_\unbranch)$ such that the following holds.  Let $X$ be a smooth
geometrically irreducible projective variety
of type $(N,r,d)$, and let $f:X \ra Y$ be a finite morphism of degree
$n$ which is generically \'etale.  Then there is an open subvariety
$U\subset Y$, affine of type $(N_\unbranch, r_\unbranch, d_\unbranch)$,
such that $f\rest{f\inv(U)}$ is \'etale.
\end{lemma}

\begin{proof}
By Lemmas \ref{lemelimination} and \ref{lembranch}, there is a collection of
homogeneous forms $\st{G_1, \cdots, G_s}$ of bounded degree on $Y$ such that $f$ is
unramified outside the vanishing locus of these forms.  Let $U$ be the
complement of the vanishing locus of the product $\prod G_i$.  Then
$f\rest{f\inv(U)}$ is \'etale, and $U$ has known type $(N_\unbranch,
r_\unbranch, d_\unbranch)$.
\end{proof}

\begin{theorem}\
\label{thnrdc}
Given data $(N,r,d)$ and $n$ there exists an effective  constant $C$ so that the
following holds.  Let $X/\ff$ be a smooth geometrically irreducible
projective variety of type $(N,r,d)$.  Then $(X/\ff,C,n)$ satisfies
\eqref{thecond}.
\end{theorem}

\begin{proof}
We adopt the notation, ideas and results of \cite{gtz}.  
Say that a variety is of known type if its type can be effectively
bounded purely in terms of the data  $(N,r,d,n)$.  Let $Y/\ff$ be a smooth
geometrically irreducible projective variety, and let $f:X \ra Y$ be a
finite, tamely ramified, generically \'etale morphism of degree $n$.
By Lemma
\ref{lemelimination}, $Y$ is of known type.

Let $\til X \ra Y$
be the Galois closure of $X \ra Y$. Then $\til X$ is a variety over some
finite extension $\til\ff$ of $\ff$, and there is a finite map of
schemes $\til f:\til X \ra Y$ of degree $\til n \le n!$ which is
tamely ramified and generically
\'etale.
Let $A = \aut(\til X/Y)$, and let $G = \aut( \til X/
Y\cross 
\til\ff_q)$ be the geometric part of the extension; then $A/G\iso
\gal(\til\ff/\ff)$.  Let $H = \aut(\til X/X)$, let $\cals$ be the
set of left cosets of $H$ in $A$, and let $\calb = \st{ a \in A :
  \ang{aG} = A/G}$. By
\cite[Lemmas 4.2 and 4.3]{gtz}, to show that $f$ is exceptional it suffices to
show that each $a\in \calb$ has at least one fixed point in $S$.

Using Lemma \ref{lemunbranched}, we may construct an affine subvariety 
$U\subset Y$ of known type such that $f\rest{f\inv(U)}$ is
unramified.  Let $\til V = \til X \cross U$; then $\til f\rest{\til
  V}: \til V \ra U$ is an \'etale Galois cover of affine varieties.
Let $\ell$ be any rational prime invertible in $\ff$.
Lemma \ref{lemgaloisbetti} gives an effective upper bound for the sum of the
compact Betti numbers $\sigma_c(\til V,\rat_\ell)$.

Let $\til W$ be any twist of $\til V$.  Then $\sigma_c(\til W,\rat_\ell) =
\sigma_c(\til V,\rat_\ell)$.  By the Lefschetz trace formula  \cite[II.3.2]{sga4h}
and Deligne's bound
for the weights of Frobenius on the compact cohomology groups of a
smooth variety \cite[Thm. 1]{deligneweil2}, we find that there is an effective constant $C$,
depending only on $(N,r,d,n)$, such that if $\ff_q/\ff$ is a finite field,
and if $q > C$, then $\til W(\ff_q)$ is nonempty.

The result now follows from the techniques of \cite{gtz}.    Fix any
$a \in \calb$.   Restricting the construction of
\cite[paragraph after 3.2]{gtz} to $\til V$, construct a certain twist
$\til V^{\til a}$ of $\til V$.    If $q >C$, then there exists some
$\til Q_a \in \til V^{\til a}(\ff_q)$.  Let $P_a = f(\til Q_a)$.  Combining
the defining property of $\til V^{\til a}$ with the fact that $\til
f\rest{\til V}$ is unramified, we have $\til Q_a/P_a$ is unramified, with
decomposition group $\ang a$.    Moreover, for this point $P_a$, the
number of points of $X(\ff_q)$ lying over $P_a$ is exactly the number of
points of $\cals$ fixed by $a$ \cite[Lemma 3.2]{gtz}.  Henceforth, suppose $q>C$.

If $f$ is surjective on $\ff_q$ points, then for each $a\in \calb$
there is at least one point of $X(\ff_q)$ lying over 
$P_a$, so that at least one point of $\cals$ is fixed by $a$. Then $f$
is exceptional \cite[Lemma 4.3]{gtz}.

Similarly, if $f$ is injective on $\ff_q$ points, then for each $a\in \calb$
there is at most one point of $\cals$ fixed by
$a$, and $f$ is again exceptional.

Finally, Lenstra's theorem \cite[Prop.\ 4.4]{gtz} shows that if $f$ is
exceptional, then it is in fact bijective on $\ff_q$-points.
\end{proof}

\begin{remark}
\label{remgtz}
Even though this statement depends on a polarization of the variety
$X$, it is still much more uniform than the best result previously
known.   For comparison's sake, note that \cite[Prop.\ 5.6]{gtz}
states that given $f:X \ra Y$ a finite separable map between normal
varieties over $\ff$, if $f(\ff_{q^m})$ is injective or surjective for
infinitely many $m$, then $f$ is exceptional.   As noted there, this
implies the existence of a 
constant $C$, depending on $f:X \ra Y$, such that $(f:X \ra Y, C,
\ff)$ satisfies \tfae.  Indeed, the proof of {\em loc. cit.} shows the
existence of a number $M$ such that if $m \ge M$, then the
surjectivity or injectivity of $f(\ff_{q^m})$ implies the
exceptionality of $f_{\ff_{q^m}}$.
\end{remark}

\begin{remark}
\label{remmoduli}
If $X$ is a member of a known family then Theorem
\ref{thnrdc} provides a uniform bound for $C$, in the following sense.

Suppose $S$ is noetherian and $\calx \ra S$ is has geometric fibers
which are smooth, projective and irreducible.
Then there
exist $(N_S, r_S, d_S)$ such that for any point $t \in S(\ff)$, the
fiber $\calx_t$ has type $(N_S, r_S, d_S)$.  Consequently, for any $n$
there exists an effective constant $C = C(S,n)$ such that if $t \in
S(\ff)$, then $(\calx_t/\ff, C, n)$ satisfies \eqref{thecond}.

Natural examples of such families $\calx \ra S$ are the tautological
families over (an open subscheme of) the Hilbert scheme of schemes of $\proj^N$ with
specified Hilbert polynomial; the moduli space of principally
polarized abelian varieties of given dimension; and the moduli space
of K3 surfaces with polarization of specified degree.  In fact, many
moduli spaces are constructed by taking the GIT quotient of an open
subscheme $S$ of a Hilbert scheme.  While the difficulties typically center
around the construction of the quotient space, the techniques of the present
paper apply directly to $S$.
\end{remark}

\section{Surfaces}

If the varieties $X$ and $Y$ are curves, then \cite[Thm.\ 4.7]{gtz} gives an
explicit bound for the constant $C$ of \ref{thnrdc} which depends only
on $n = \deg(f)$ and the genus of $X$.  In this section, we 
show that if
$X$ is a surface of general type, then there
is a constant $C$ which depends only on the Hodge numbers of $X$ and
on $n$ such that $(X/\ff,C,n)$ satisfies \eqref{thecond}.  If $X$
lifts to $W(\ff)/p^2$, we will deduce that $C$ need only depend on
$\deg(f)$ and on the $\ell$-adic Betti numbers of $X$. 

Let $k$ be any field.  If $X/k$ is a projective surface, we denote its Hodge numbers by 
$h^{ij}(X) = \dim H^j( X, \Omega^i_X)$.

\begin{lemma}
\label{lemhodge}
If $X/k$ is a smooth projective surface of general type with specified Hodge numbers
$h^{ij}(X) = h^{ij}$, then there is a bound for the type
of $X$ which depends only on $h^{ij}$.
\end{lemma}

\begin{proof}
First, we prove the result under the additional assumption that $X$ is
minimal.
Since the Hodge numbers of $X$ are known, in particular one knows
$\chi(X,\calo_X) = \sum_i (-1)^i h^{0,i}(X)$ and 
$K_X^2 = 12 \chi(X.\calo_X) - \sum(-1)^{i+j}h^{ij}(X)$.  Since $X$ is
of general type, $5K_X$ is very ample.  Let $N = 10 K_X^2 +
\chi(\calo_X) - 1$.  Then $\phi_{5K_X}: X \ra \proj^N$ is birational
onto its image $X_0$.  The embedded surface $X_0 \subset \proj^N$ has
Hilbert polynomial $h_{X_0}(T) = (25/2) (K_X^2) T^2 - (5/2) K_X^2 T +
\chi(\calo_X)$, and thus is of known type.
Moreover, $X_0$ is normal, with at worst Du Val
singularities corresponding to the contraction of $(-2)$ curves on
$X$.  The number $m$ of such curves may be bounded in terms of the Hodge
numbers of $X$ \cite[p. 614]{mumfordcanonical}.  Since $X$ is obtained
from $X_0$ by at most $m$ blowups, $X$ has known type.

Finally, we prove the result for arbitrary smooth projective surfaces
of general type.
If $X$ is such a surface, and if $\pi:X \ra \bar X$ is the
blowing-down of a $(-1)$-curve, then $h^{1,1}(X) = h^{1,1}(\bar X)+1$,
while $h^{ij}(X) = h^{ij}(\bar X)$ for all other $(i,j)$.  Therefore,
$X$ differs from its minimal model $X_{\min}$ 
by at most $h^{1,1}(X)$ blow-ups, and the Hodge numbers of $X_{\min}$ are
known.  Since a variety obtained by a bounded number of blowups from a
variety of known type is again of known type, $X$ has known type.
\end{proof}

\begin{corollary}
\label{corsurf}
Given nonnegative integers $b_1$, $b_2$ and $b_3$ and a natural number
$n$, there exists an
effective constant $C$ so that the following holds.  Let $X$ be a
smooth, projective, geometrically irreducible surface over a finite
field $\ff$.  Suppose
that $X$ is of general type and admits a flat lifting to $W(\ff)/p^2$, and that $\dim
H^i(X,\rat_\ell) = b_i$ for $1 \le i \le 3$.  Then $(X/\ff,C,n)$ satisfies
\eqref{thecond}.
\end{corollary}

\begin{proof}
  The hypothesis that $X$ lifts modulo $p^2$ implies that the Hodge to
  deRham sequence for $X$ degenerates \cite{deligneillusie}; this,
  combined with the comparison theorem between \'etale and deRham
  cohomology, implies that $\sum_{r=0}^i h^{i,i-r}(X) =
  b_i$ for each $i$.  Therefore, there are only finitely many
  possibilities for the Hodge numbers of $X$, and thus for the type of
  $X$ (Lemma \ref{lemhodge}).  The result now follows from Theorem
  \ref{thnrdc}.
\end{proof}

Similar results are possible for surfaces which are not necessarily of
general type.  Here is one example.

\begin{corollary}
Given a natural number $n$, there exists a constant $C$ so that the
following holds.  Let $\ff$ be a finite field of odd characteristic,
and let $X/\ff$ be an Enriques surface.  Then $(X/\ff,C,n)$ satisfies
\eqref{thecond}.
\end{corollary}

\begin{proof}
By \cite[Introduction]{cossec83}, $X_{\bar\ff}$ admits an \'etale
double cover which is the intersection of three quadrics in
$\proj^5$; therefore, $X_{\bar\ff}$ is of known type.  For each scheme
$Z$ which arises in the proof of Theorem \ref{thnrdc}, $Z_{\bar\ff}$
is of known type.  Therefore, the conclusion of \ref{thnrdc} applies
to $X$.
\end{proof}

\section{Examples}

Most known examples of exceptional covers involve curves, especially
the projective line. Higher-dimensional examples tend to involve
special varieties, such as abelian varieties or projective spaces.
While exceptional covers are indeed rare, in this section we show that
there actually exist infinitely many exceptional covers of each
dimension over a given finite field.     See also forthcoming work of
Lenstra, Moulton and Zieve.

Theorem \ref{thnrdc} (like its antecedents in \cite{gtz})
states that if a finite field is sufficiently large
relative to the topology of two varieties, then a cover is injective
on rational points if and only if it is surjective.  We give examples
showing that this fails if the field is not sufficiently large.

Throughout this section, let $\ff = \ff_{q_0}$ be a finite field of
cardinality $q_0$.

\begin{proposition}
\label{propex}
Let $Y/\ff$ be a smooth projective geometrically irreducible variety.  
\begin{alphabetize}
\item There exists an exceptional cover $f_\exc:X_\exc \ra Y$.
\item There exists a cover $f_\surj:X_\surj \ra Y$ such that $X_\surj(\ff) \ra
  Y(\ff)$ is surjective but not injective.
\item Suppose $Y(\ff)$ is nonempty.  There exists a cover $f_\inj:X_\inj \ra
  Y$ such that $X_\inj(\ff)\ra Y_\inj(\ff)$ is injective but not surjective.
\end{alphabetize}
Each of $X_\exc$, $X_\surj$ and $X_\inj$ is a projective smooth
geometrically irreducible variety, and each of $f_\exc$, $f_\surj$ and
$f_\inj$ is a finite, generically \'etale surjective morphism.
\end{proposition}

We present constructions after recalling some of the technology for
producing space-filling and space-avoiding varieties developed by
Poonen in \cite{poonenbertini}.

\begin{lemma}
\label{lemspacefill}
Let $f:X \ra Y$ be a surjective morphism of smooth, projective
geometrically irreducible varieties over $\ff$ which is smooth over an
open subset of $Y$ containing $Y(\ff)$.
Then there exists a smooth projective geometrically irreducible
subvariety $Z\subset X$ such that $Z(\ff) = X(\ff)$, and $Z \inject X
\ra Y$ is a surjective morphism which is \'etale over an open subset of
$Y$ containing $Y(\ff)$.
\end{lemma}

\begin{proof} 
  Fix an embedding $X \inject \proj^M$.  By induction on $r := \dim X -
  \dim Y$, it suffices to show that there is a hypersurface $H \subset
  \proj^M$ such that $Z := H \cap X$ is smooth, projective and
  geometrically irreducible, $Z(\ff) = X(\ff)$, and $f\rest{Z}:Z \ra
  Y$ is smooth of relative dimension $r-1$ over an open subset of $Y$
  which contains $Y(\ff)$.

  We follow the proof of \cite[Thm.\ 3.3]{poonenbertini}, and describe
  suitable hypersurfaces in terms of local tangency conditions.  In
  this description, all intersections are scheme-theoretic, and the
  empty scheme is smooth of any dimension.  Let $S = Y(\ff)$; if $S$
  is empty, let $S$ consist of a point $Q \in Y(\bar\ff)$ such that
  $f\rest{f\inv(Q)}$ is smooth.
  
  For each $P \in X(\ff)$, choose a codimension one subspace $V_P \subset
  T_{P,\proj^M}$ such that $V_P\cap T_{P,X}$ has codimension one in
  $T_{P,X}$, and the induced map $(df)_P: (V_P\cap T_{P,X}) \ra
  T_{f(P),Y}$ is surjective.  Consider the problem of finding a
  hypersurface $H\subset \proj^M$ such that for each $P\in X(\ff)$, $P\in
  H$ and $T_{P,H} = V_P$; and for each other closed point $P$ of $X$,
  $H$ and $H\cap X$ are smooth of dimensions $M-1$ and $\dim X - 1$,
  respectively, at $P$.  Refine this problem by insisting that for 
  each closed point $P$ of $X$ and each $Q\in S$,
 the intersection $H\cap X \cap f\inv(Q)$ is smooth of
  dimension $r-1$ at $P$.  

  Then \cite[Thm.\ 1.3]{poonenbertini} guarantees the existence of a
  smooth geometrically irreducible hypersurface which satisfies these
  conditions.  Choose such a hypersurface $H$, and let $Z = H\cap X$.
  Then $Z$ is smooth, $Z(\ff) = X(\ff)$, and the morphism $Z \inject X
  \ra Y$ is generically smooth of relative dimension $r-1$, and in
  particular smooth  over each
  $\ff$-rational point of $Y$.   A dimension count shows the morphism is dominant, and thus surjective.
\end{proof}

There is a point-avoiding complement to Lemma \ref{lemspacefill}:

\begin{lemma}
\label{lemspaceavoid}
Let $f:X \ra Y$ be a surjective morphism of smooth, projective
geometrically irreducible varieties over $\ff$ which is smooth over
an open subset of $Y$ containing $Y(\ff)$.
Then there exists a smooth projective geometrically irreducible
subvariety $Z\subset X$ such that $Z(\ff)$ is empty, and $Z \inject X
\ra Y$ is a surjective morphism which is \'etale over an open subset of
$Y$ containing $Y(\ff)$.
\end{lemma}

\begin{proof}
This is a relative version of \cite[Cor.\ 3.6]{poonenbertini}.  The
proof is the same as that of Lemma \ref{lemspaceavoid}, except that
in the local
conditions we insist that the hypersurface avoid each point of
$X(\ff)$.
\end{proof}

With these results secured, construction of examples is straight-forward.

\begin{proof}[Proof of Proposition \ref{propex}]

  For part (a), let $Y_0 = \spec B$ be an open affine subvariety of
  $Y$.  Fix a prime $\ell$ relatively prime to $q_0-1$, and let $u\in
  B$ be an element of some system of uniformizing parameters for
  $Y_0$.  Let $A = B[x]/(x^\ell-u)$, and let $X_0 = \spec A$.  Then
  $X_0$ is geometrically irreducible (since $x^\ell-u$ is irreducible
  over $B\otimes \bar\ff$) and smooth (by the Jacobian criterion).
  Let $\ff_q/\ff$ be any finite extension with $\gcd(q-1,\ell) =1$,
  and suppose $Q \in Y_0 (\ff_q)$.  Then there is a unique solution in
  $\ff_q$ to the equation $x^\ell = u(Q)$, and $X_0(\ff_q) \ra
  Y_0(\ff_q)$ is bijective.  Since this is true for arbitrarily large
  $q$, the cover $X_0 \ra Y_0$ is exceptional.  Let $X$ be the
  normalization of $Y$ in ${\rm Frac}(A)$; it, too, is smooth.  The
  geometric characterization of exceptionality makes it clear that
  being exceptional is a birational property, and so $X \ra Y$ is
  exceptional.

For (b), let $V$ be a smooth projective geometrically irreducible
variety such that $\abs{V(\ff)} \ge 2$; concretely, one may take $V =
\proj^1$.  Then the product $Y \cross V$ is again smooth projective
and geometrically irreducible, and the projection $Y\cross V \ra Y$ is
smooth.  Moreover, $(Y\cross V)(\ff) \ra Y(\ff)$ is surjective but not
injective.  By Lemma \ref{lemspacefill}, there is a smooth projective
geometrically irreducible subvariety $\til X_\surj \subset Y\cross V$ such
that $\til X_\surj(\ff) = (Y\cross V)(\ff)$, and $f\rest{\til X_\surj}$ is
\'etale over each element of $Y(\ff)$.  Consider the Stein
factorization $\til X_\surj \sra s X_\surj \sra t Y$ of the projective
morphism $f\rest{\til
  X_\surj}$.  The morphism $t$ is finite and generically \'etale.
Moreover, since $\til X_\surj \ra Y$ is finite over $Y(\ff)$, the
birational morphism $s$ induces a bijection $\til X_\surj(\ff) \sriso
X_\surj(\ff)$.  Therefore, $X_\surj \ra Y$ is finite and generically
\'etale, and $X_\surj(\ff) \ra Y(\ff)$ is surjective but not injective.

The proof of (c) is similar, except that we use Lemma
\ref{lemspaceavoid} to construct a suitable subvariety $\til X_\inj \subset
Y\cross V$ with $\til X_\inj(\ff)$ empty.  Again, the Stein
factorization $\til X_\inj \ra X_\inj \ra Y$ yields a finite cover of
$Y$, and $\til X_\inj(\ff) \sriso X_\inj(\ff)$.
If $Y(\ff)$ is nonempty, then
$X_\inj(\ff) \ra Y(\ff)$ is injective but not surjective.
\end{proof}

\begin{remark}
The construction of the ``space-filling'' variety $X_\surj \subset Y\cross V$ 
in Proposition \ref{propex}.(b) depends on $\ff$; the equality of sets
$X_\surj(\ff) = (Y\cross V)(\ff)$ is only possible if $q_0$ is small
relative to the Betti numbers of $X_\surj$.  Similarly, in
\ref{propex}.(c), the variety $X_\inj$ acquires rational points over
sufficiently large extensions of $\ff$.  While there is no reason to
believe that the constant $C$ in Theorem \ref{thnrdc} is optimal, these examples
indicate that the equivalences in  \tfae\ cannot hold for $\ff$ itself,
but only for sufficiently large extensions.
\end{remark}

\bibliographystyle{abbrv} 
\bibliography{jda}

\end{document}